\documentclass[11pt,letterpaper,reqno]{amsart}
\usepackage{tikz}
\usetikzlibrary{positioning, shapes.geometric, arrows.meta, calc, positioning}
\usepackage{amssymb}
\usepackage{amsmath}
\usepackage{amsthm}
\usepackage{amsfonts}
\usepackage{bbm}
\usepackage{enumitem} 
\usepackage{pgfplots}
\pgfplotsset{compat=1.18} 
\usepackage{booktabs}
\usepackage{graphicx}
\usepackage[T1]{fontenc}
\usepackage{doi}
\usepackage{float} 
\usepackage{verbatim}
\usepackage{bookmark}
\usepackage{hyperref}
\hypersetup{pdfstartview={FitH}}

\usepackage{xcolor}
\usepackage{colortbl}

\usepackage{cleveref}

\newtheorem{theorem}{Theorem}[section]
\newtheorem{lemma}[theorem]{Lemma}
\newtheorem{proposition}[theorem]{Proposition}

\newtheorem{conjecture}[theorem]{Conjecture}
\theoremstyle{definition}

\newtheorem{remark}[theorem]{Remark}

\numberwithin{equation}{section}

\newcommand{\eps}{\varepsilon}

\newcommand{\Leg}[2]{\left(\frac{#1}{#2}\right)}

\newcommand{\piodd}{\pi_{\mathrm{odd}}}

\begin{document}

\title[First-passage times for character sums]{Average first-passage times for character sums}

\author[Q.~Tang]{Quanyu Tang}
\author[H.~Zhang]{Hao Zhang}

\address{School of Mathematics and Statistics, Xi'an Jiaotong University, Xi'an 710049, P. R. China}
\email{tang\_quanyu@163.com}

\address{School of Mathematics, Hunan University, Changsha 410082, P. R. China}
\email{zhanghaomath@hnu.edu.cn}



\begin{abstract}
Let $\varepsilon>0$ and, for an odd prime $p$, set
$$
S_\ell(p):=\sum_{n\le \ell}\left(\frac{n}{p}\right).
$$
Define the \emph{first-passage time}
$$
f_\varepsilon(p):=\min\{\ell\ge 1:\ S_\ell(p)<\varepsilon\ell\}.
$$
We prove that there exists a constant $c_\varepsilon>0$ such that, as $x\to\infty$,
$$
\sum_{p\le x} f_\varepsilon(p)\sim c_\varepsilon \frac{x}{\log x}.
$$
\end{abstract}

\maketitle

\section{Introduction}

A central question in number theory is to understand the behavior of the quadratic character sum
\[
S_\ell(p):=\sum_{n\le \ell}\Leg{n}{p},
\]
where $\Leg{\cdot}{p}$ denotes the Legendre symbol modulo an odd prime $p$.
Such sums encode fundamental information on the distribution of quadratic residues and non-residues
modulo $p$. In particular, bounds for $S_\ell(p)$ lead to results on the size of the least quadratic
non-residue modulo $p$; for a more detailed discussion and further background, see
\cite[Section~1]{HL}.

Let $\chi_p(n):=\Leg{n}{p}$. A basic heuristic models the values $\chi_p(n)$ as independent unbiased
$\pm 1$ signs, in which case $S_\ell(p)$ is typically of size $\asymp \sqrt{\ell}$ and, for any fixed
$\varepsilon>0$, the inequality $S_\ell(p)<\varepsilon \ell$ should hold for all sufficiently large $\ell$.
This motivates quantitative questions about the time at which the graph of $S_\ell(p)$ falls below
the linear barrier $\varepsilon\ell$.

In~\cite[Eq.~(80)]{Er65b}, Erd\H{o}s posed the following conjecture, which appears as Problem~\#981 on Bloom's
Erd\H{o}s Problems website~\cite{EP981}.\footnote{An earlier version of the Problem~\#981 page~\cite{EP981} misstated the question by
defining a \emph{first-passage} threshold (asking for the smallest $\ell$ with $S_\ell(p)<\varepsilon\ell$),
rather than the \emph{eventual-time} threshold stated below. The present paper studies and resolves that
earlier first-passage formulation.}

\begin{conjecture}[Erd\H{o}s's eventual-time problem]\label{conj:stronger}
Let $\varepsilon>0$, and for each odd prime $p$ let $F_\varepsilon(p)$ be the smallest integer such that
\begin{equation}\label{eq:intro:Feps}
S_\ell(p)<\varepsilon\ell\qquad\text{for every }\ell\ge F_\varepsilon(p).
\end{equation}
Then there exists a constant $C_\varepsilon>0$ such that
\[
\sum_{p \leq x}F_\varepsilon(p)\sim C_\varepsilon\,\frac{x}{\log x}.
\]
\end{conjecture}
This conjecture was proved by Elliott \cite{El69}.

Motivated by Conjecture~\ref{conj:stronger}, we consider the associated \emph{first-passage} threshold.  For $\varepsilon>0$ and each odd
prime $p$, define
\[
f_\varepsilon(p):=\min\{\ell\ge 1:\ S_\ell(p)<\varepsilon\ell\}.
\]
Clearly one has $f_\varepsilon(p)\le F_\varepsilon(p)$ for each $p$, since \eqref{eq:intro:Feps} forces
$S_{F_\varepsilon(p)}(p)<\varepsilon F_\varepsilon(p)$.
However, an asymptotic formula for $\sum_{p\le x}F_\varepsilon(p)$ does not by itself imply the corresponding
asymptotic for $\sum_{p\le x}f_\varepsilon(p)$.
Our main result establishes the expected asymptotic for $\sum_{p\le x}f_\varepsilon(p)$.

\begin{theorem}\label{thm:intro:first}
For every $\varepsilon>0$ there exists a constant $c_\varepsilon\in(0,\infty)$ such that, as $x\to\infty$,
\[
\sum_{p\le x} f_\varepsilon(p)\sim c_\varepsilon\,\frac{x}{\log x}.
\]
\end{theorem}

\subsection*{Proof outline}
Write the average first-passage time in terms of tail probabilities:
\[
\frac{1}{\pi_{\mathrm{odd}}(x)}\sum_{p\le x} f_\varepsilon(p)
=\sum_{m\ge 0} a_m(x),
\qquad
a_m(x):=\frac{1}{\pi_{\mathrm{odd}}(x)}\#\{p\le x:\ f_\varepsilon(p)>m\},
\]
where $\pi_{\mathrm{odd}}(x)=\#\{p\le x:\ p\ \text{is an odd prime}\}$.
For each fixed $m$, the event $\{f_\varepsilon(p)>m\}$ is determined by the finite vector
$\left(\chi_p(q)\right)_{2 \le q\le m}$, and the prime number theorem for Dirichlet characters gives
the pointwise limits $a_m(x)\to \widehat a_m$ as $x\to\infty$.
To pass from these pointwise limits to an asymptotic for the average, the key additional input is a uniform tail bound, which allows us to interchange the limit $x\to\infty$
with the infinite sum over $m$.
We obtain this by dominating $a_m(x)$ by the density of primes with $S_m(p)\ge \varepsilon m$ and controlling
this density via high-moment estimates for $S_m(p)$: for $m\le x^{1/6-\kappa}$ we use a sixth-moment bound
coming from the quadratic large sieve, while for larger $m$ we combine quadratic reciprocity,
Heath--Brown's large sieve, and the P\'olya--Vinogradov inequality.

\section{Preliminaries}\label{sec:Preliminaries}

\subsection{Notation and conventions}

Fix $\varepsilon>0$. The case $\varepsilon > 1$ is trivial since $S_1(p)=1<\varepsilon$ for every odd prime
$p$, hence $f_\varepsilon(p)=1$. Thus, throughout the paper we assume that $0<\varepsilon \leq 1$, with the convention $S_0(p):=0$.


We use Vinogradov's asymptotic notation. For functions \(f=f(n)\) and \(g=g(n)\), we write
\(f=O(g)\), \(f\ll g\), or \(g\gg f\) to mean that there exists a constant \(C>0\) such that
\(|f(n)|\le C g(n)\) for all sufficiently large \(n\).
We write \(f\asymp g\) to mean that \(f\ll g\) and \(g\ll f\), and we write
\(f=o(g)\) to mean that \(f(n)/g(n)\to 0\) as \(n\to\infty\).
If the implicit constant is allowed to depend on one or more parameters $z_1,\dots,z_r$, we indicate this by
writing $f\ll_{z_1,\dots,z_r} g$, $g\gg_{z_1,\dots,z_r} f$ or $f=O_{z_1,\dots,z_r}(g)$.
Finally, we write $f\sim g$ to mean that $f(n)/g(n)\to 1$ as $n\to\infty$.

Throughout the paper, the letter $p$ always denotes an odd prime.
We use the standard prime-counting function $\pi(x)=\#\{p\le x:\ p\ \text{prime}\}$, and we set
\[
\piodd(x):=\#\{p\le x:\ p\ \text{is an odd prime}\}=\pi(x)-1\qquad(x\ge 2).
\]
Accordingly, $\sum_{p\le x}$ denotes a sum over odd primes $p\le x$, and we normalize by $\piodd(x)$ when taking averages.
We denote by $\mathcal{F}(x)$ the set of fundamental discriminants $d$ such that $|d|\le x$.
We write $m=\square$ to indicate that $m$ is a perfect square.




\subsection{Two forms of the quadratic large sieve}

Next we present two forms of the quadratic large sieve. The first is a large sieve inequality for prime
discriminants, which is a special case of~\cite[Lemma~9]{MoVa79}; see also~\cite[Lemma~4.8]{HL}.

\begin{lemma}[\cite{MoVa79}]\label{lem:LS-prime}
Let $\kappa>0$ be a fixed small number. Let $x\ge 2$ and $2\le N\le x^{1/2-\kappa}$. Then for arbitrary
complex numbers $a_1,\dots,a_N$ we have
\[
\sum_{p\le x}\Bigl|\sum_{n\le N} a_n \Leg{n}{p}\Bigr|^2
\ll_\kappa \frac{x}{\log x}\sum_{\substack{m,n\le N\\ mn=\square}}|a_m a_n|.
\]
\end{lemma}

The second form is due to Heath--Brown~\cite[Corollary~2]{HB95}; see also~\cite[Lemma~4.1]{La24}.

\begin{lemma}[\cite{HB95}]\label{lem:LS-HB}
Let $x,N\ge 2$. Then for arbitrary complex numbers $a_1,\dots,a_N$ and for any $\rho>0$ we have
\[
\sum_{d\in\mathcal{F}(x)}\Bigl|\sum_{n\le N} a_n \Leg{d}{n}\Bigr|^2
\ll_\rho (xN)^\rho (x+N)\sum_{\substack{m,n\le N\\ mn=\square}}|a_m a_n|.
\]Here $\Leg{d}{\cdot}$ denotes the Kronecker symbol modulo $|d|$.
\end{lemma}

\subsection{Cubic coefficients and a square-condition bound}

For an integer $\ell\ge 1$, define
\begin{equation}\label{eq:cu-def}
c_u(\ell):=\#\{(a,b,c)\in\{1,\dots,\ell\}^3:\ abc=u\}\qquad(1\le u\le \ell^3),
\end{equation}
and set
\begin{equation}
C_3(\ell):=\sum_{\substack{u,v\le \ell^3\\ uv=\square}} c_u(\ell)c_v(\ell).
\end{equation}

\begin{lemma}\label{lem:C3m-bound}
For every $\delta>0$ one has
\[
C_3(m)\ll_\delta m^{3+\delta}.
\]
\end{lemma}

\begin{proof}

We use the trivial bound $c_u(m)\le d_3(u)$, where 
\[d_k(u)=\#\{(a_1,\dots,a_k)\in \mathbb{N}^k\ :a_1\dotsm a_k=u\}.\]
To bound $d_3(u)$ in terms of the usual divisor function $d(u)=d_2(u)$. Write
\[
u=\prod_{q\ \mathrm{prime}}q^{s_q}.
\]
Then
\[
d_3(u)
=\prod_{q\mid u}\#\{(e_1,e_2,e_3)\in\mathbb{N}^3:\ e_1+e_2+e_3=s_q\}
=\prod_{q\mid u}\binom{s_q+2}{2}.
\]
Since $\binom{s+2}{2}=\frac{(s+1)(s+2)}{2}\le (s+1)^2$ for every $s\ge0$, we obtain
\[
d_3(u)\le \prod_{q\mid u}(s_q+1)^2=d(u)^2\ll_{\theta} u^{2\theta}.
\]
The last inequality follows from a classical result that $d(n)\ll_{\theta} n^{\theta}$ (see e.g.\ \cite[Theorem 315]{HW}). Thus $c_u(m)\ll_\theta u^{2\theta}$.

The condition $uv=\square$ is equivalent to writing $u=ds^2$ and $v=dt^2$ with $d$ squarefree.
Therefore,
\[
C_3(m)\le \sum_{\substack{d\ \mathrm{squarefree}\\ d\le m^3}}
\Bigl(\sum_{ds^2\le m^3} c_{ds^2}(m)\Bigr)^2
\ll_\theta \sum_{d\le m^3}\Bigl(\sum_{s\le m^{3/2}/\sqrt d} (ds^2)^{2\theta}\Bigr)^2.
\]
The inner sum satisfies
\[
\sum_{s\le m^{3/2}/\sqrt d} (ds^2)^{2\theta}
= d^{2\theta} \sum_{s\le m^{3/2}/\sqrt d} s^{4\theta}
\leq d^{2\theta}\Bigl(\frac{m^{3/2}}{\sqrt d}\Bigr)^{1+4\theta}
= m^{\frac32(1+4\theta)}d^{-\frac{1}{2}}.
\]
Hence
\[
C_3(m)\ll_\theta \sum_{d\le m^3}\left( m^{\frac32(1+4\theta)} d^{-\frac{1}{2}} \right)^2
= m^{3+12\theta}\sum_{d\le m^3} d^{-1}
\ll_\theta m^{3+12\theta}\log m.
\]
Now choose $\theta=\delta/24$ and use $\log m\ll_\delta m^{\delta/2}$ to conclude $C_3(m)\ll_\delta m^{3+\delta}$.
\end{proof}

\begin{remark}
In fact, de la Bret\`eche--Kurlberg--Shparlinski~\cite{BKS21} prove a substantially
stronger result than Lemma~\ref{lem:C3m-bound}.
Expanding the definition shows that
\[
C_3(m)
=\#\Bigl\{(x_1,\dots,x_6)\in[1,m]^6:\ x_1x_2x_3x_4x_5x_6 \ \text{is a perfect square}\Bigr\}.
\]
Hence \(C_3(m)=N_6(m)\) in the notation of~\cite{BKS21}. Therefore, substituting
\(n=6\) into \cite[Corollary~2.4]{BKS21} yields
\[
C_3(m)\ll m^3(\log m)^{15}.
\]
\end{remark}

Next we give a uniform bound for $f_\eps(p)$ via the P\'olya--Vinogradov inequality.

\begin{lemma}\label{lem:PV-cutoff}
There exists an absolute constant $C>0$ such that for every odd prime $p$ and every $\ell\ge 1$,
\[
|S_\ell(p)| < C\sqrt{p}\log p.
\]
Consequently,
\[
f_\eps(p)\le \frac{C}{\eps}\sqrt{p}\log p
\]
for all odd primes $p$.
\end{lemma}

\begin{proof}
The P\'olya--Vinogradov inequality (see, for example \cite[Theorem~9.18]{MV}), states that for any nonprincipal Dirichlet character $\chi$ modulo $q$,
\[
\max_{1\le t\le q}\Bigl|\sum_{n\le t}\chi(n)\Bigr|\ll \sqrt{q}\log q.
\]
Applying this to $\chi=\chi_p$ gives $|S_\ell(p)|\ll \sqrt{p}\log p$ for $1\le \ell\le p$.
For $\ell>p$, the character $\chi_p$ is periodic modulo $p$ and $\sum_{n=1}^{p}\chi_p(n)=0$, so
$S_\ell(p)=S_{\ell\bmod p}(p)$ and the same bound holds. Hence, there exists an absolute constant $C>0$ such that for every odd prime $p$ and every $\ell\ge 1$, $|S_\ell(p)|< C\sqrt{p}\log p$. For any $\ell \geq (C/\eps)\sqrt{p}\log p$, we have
$S_\ell(p)\le |S_\ell(p)|<\eps \ell$, so  $f_\eps(p)\le \frac{C}{\eps}\sqrt{p}\log p$.
\end{proof}

\subsection{A density parameter}

Now, we need to estimate the density of primes such that $S_\ell(p)\ge \eps \ell$. For this purpose, for $x\ge 2$ and each integer $\ell\ge 1$, we define
\begin{equation}
P_{\ell}(x):=\frac{1}{\piodd(x)}\#\{p\le x:\ S_\ell(p)\ge \eps \ell\}.
\end{equation}

We give the estimation of $P_{\ell}(x)$ through the sixth-moment of $S_{\ell}(p)$. 

\begin{lemma}\label{thm:boundofam}
Fix $\kappa\in (0,1/6)$ and $\delta\in (0,1)$. When $2 \le \ell \le x^{1/6-\kappa}$, we have the uniform bound
\begin{equation}\label{eq:am-smallm}
P_\ell(x)\ll_{\eps,\kappa,\delta} \ell^{-3+\delta}.
\end{equation}
When $\ell>x^{1/6-\kappa}$, we have the uniform bound
\begin{equation}\label{eq:am-HB}
P_\ell(x)\ll_{\eps,\rho,\delta} (\log x)\,x^{\rho-1}\,\ell^{3\rho-3+\delta}\,(x+\ell^3).
\end{equation}
\end{lemma}

\begin{proof}

Using the following trivial bound, we transform the problem to the estimation of sixth-moment of $S_{\ell}(p)$.
\begin{equation}\label{eq:am-bound-general}
P_{\ell}(x)\le \frac{1}{\eps^6 {\ell}^6\,\piodd(x)}\sum_{p\le x} |S_{\ell}(p)|^6.
\end{equation}By~\eqref{eq:cu-def}, we have the identity
\[
S_{\ell}(p)^3=\sum_{a,b,c\le {\ell}}\Leg{abc}{p}
=\sum_{u\le {\ell}^3} c_u(\ell) \Leg{u}{p}.
\]
Therefore
\begin{equation}\label{eq:Sm6-as-quadratic}
|S_{\ell}(p)|^6=\Bigl|\sum_{u\le {\ell}^3} c_u(\ell) \Leg{u}{p}\Bigr|^2.
\end{equation}

\smallskip
\noindent\emph{Case 1: $2 \leq {\ell}\le x^{1/6-\kappa}$.}
Then ${\ell}^3\le x^{1/2-3\kappa}$. Applying Lemma~\ref{lem:LS-prime} with $N={\ell}^3 \geq 2$, we obtain
\[
\sum_{p\le x} |S_{\ell}(p)|^6
=\sum_{p\le x}\Bigl|\sum_{u\le {\ell}^3} c_u(\ell) \Leg{u}{p}\Bigr|^2
\ll_\kappa \frac{x}{\log x}\sum_{\substack{u,v\le {\ell}^3\\ uv=\square}} c_u(\ell) c_v(\ell)
=\frac{x}{\log x} C_3({\ell}).
\]
By Lemma~\ref{lem:C3m-bound}, $C_3({\ell})\ll_\delta {\ell}^{3+\delta}$, hence
\begin{equation}\label{eq:Sm6prime}
\sum_{p\le x}|S_{\ell}(p)|^6\ll_{\kappa,\delta}\frac{x}{\log x} {\ell}^{3+\delta}.
\end{equation}
Since $\piodd(x)\asymp x/\log x$, inserting \eqref{eq:Sm6prime} into \eqref{eq:am-bound-general} yields
\eqref{eq:am-smallm}.

\smallskip
\noindent\emph{Case 2: ${\ell}>x^{1/6-\kappa}$.}
Let $\rho>0$ be a small parameter. For an odd prime $p$, define $p^\ast:=(-1)^{(p-1)/2}p$.
Quadratic reciprocity implies $\Leg{u}{p}=\Leg{p^\ast}{u}$ for all integers $u$.
Thus
\[
\sum_{p\le x}|S_{\ell}(p)|^6
=\sum_{p\le x}\Bigl|\sum_{u\le {\ell}^3} c_u(\ell) \Leg{p^\ast}{u}\Bigr|^2
\le \sum_{d\in\mathcal{F}(x)}\Bigl|\sum_{u\le {\ell}^3} c_u(\ell) \Leg{d}{u}\Bigr|^2.
\]
Applying Lemma~\ref{lem:LS-HB} with $N={\ell}^3$, we obtain
\begin{align}
\sum_{p\le x}|S_{\ell}(p)|^6
&\ll_\rho (x {\ell}^3)^\rho (x+{\ell}^3)\sum_{\substack{u,v\le {\ell}^3\\ uv=\square}} c_u(\ell) c_v(\ell)
=(x {\ell}^3)^\rho (x+{\ell}^3)\,C_3({\ell})\nonumber\\
&\ll_{\rho,\delta} (x {\ell}^3)^\rho (x+{\ell}^3) {\ell}^{3+\delta}.
\label{eq:Sm6HB}
\end{align}
Combining \eqref{eq:Sm6HB} with \eqref{eq:am-bound-general} and $\piodd(x)\asymp x/\log x$ gives \eqref{eq:am-HB}.
\end{proof}

As a consequence, we obtain a uniform upper bound for $P_{\ell}(x)$.

\begin{theorem}\label{lem:Px-large-o1}
Fix $\kappa\in (0,\frac{1}{6})$. Then for any positive real number $\beta<\frac{1}{3}-2\kappa$, we have 
\begin{equation}\label{eq:Px-power-saving-claim}
\sum_{\ell> x^{1/6-\kappa}} P_{\ell}(x)\ \ll_{\varepsilon,\kappa,\beta}\ x^{-\beta}.
\end{equation}
\end{theorem}

\begin{proof}
    Let $C>0$ be the absolute constant from Lemma~\ref{lem:PV-cutoff}. For primes $p\le x$ and integers $\ell \geq(C/\eps)\sqrt{x}\log x$, Lemma~\ref{lem:PV-cutoff} gives
$|S_\ell(p)|< C\sqrt{x}\log x \leq \varepsilon\ell$, hence $P_{\ell}(x)=0$. Thus
\[
\sum_{\ell> x^{1/6-\kappa}} P_{\ell}(x)=\sum_{x^{1/6-\kappa}<\ell\le (C/\eps)\sqrt{x}\log x} P_{\ell}(x).
\]For any $\ell\ge x^{1/6-\kappa}$, \eqref{eq:am-HB} yields the bound

\begin{equation}\label{eq:Px-HB-pointwise-UI}
P_{\ell}(x)\ll_{\varepsilon,\rho,\delta}
(\log x)\,x^{\rho-1}\,\ell^{3\rho-3+\delta}\,(x+\ell^3)
\qquad(\ell\ge 2).
\end{equation}

Split the summation at $\ell=x^{1/3}$.
\smallskip

\noindent\emph{Range I: $x^{1/6-\kappa}<\ell\le x^{1/3}$.}
Here $x+\ell^3\le 2x$, so \eqref{eq:Px-HB-pointwise-UI} gives
\[
P_{\ell}(x)\ll_{\varepsilon,\rho,\delta} (\log x)\,x^\rho\,\ell^{3\rho-3+\delta}.
\]
Assuming $3\rho-2+\delta<0$, we obtain
\[
\sum_{x^{1/6-\kappa}<\ell\le x^{1/3}}P_{\ell}(x)
\ll_{\varepsilon,\rho,\delta} (\log x)\,x^\rho \sum_{\ell>x^{1/6-\kappa}}\ell^{3\rho-3+\delta}
\ll_{\varepsilon,\rho,\delta} (\log x)\,x^{\,\rho+(1/6-\kappa)(3\rho-2+\delta)}.
\]

\noindent\emph{Range II: $x^{1/3}<\ell\le \frac{C}{\eps}\sqrt{x}\log x$.}
Here $x+\ell^3 \leq 2\ell^3$, so \eqref{eq:Px-HB-pointwise-UI} gives
\[
P_{\ell}(x)\ll_{\varepsilon,\rho,\delta} (\log x)\,x^{\rho-1}\,\ell^{3\rho+\delta}.
\]
Therefore,
\[
\sum_{x^{1/3}<\ell\le \frac{C}{\eps}\sqrt{x}\log x}P_{\ell}(x)
\ll_{\varepsilon,\rho,\delta} (\log x)\,x^{\rho-1}\,(\sqrt{x}\log x)^{1+3\rho+\delta}
\ll_{\varepsilon,\rho,\delta} x^{-1/2+(5/2)\rho+\delta/2}\,(\log x)^{2+3\rho+\delta}.
\]

Now choose $\rho,\delta>0$ sufficiently small (depending on $\kappa,\beta$) so that
\[
\rho+(1/6-\kappa)(3\rho-2+\delta)< -\beta
\qquad\text{and}\qquad
-1/2+(5/2)\rho+\delta/2< -\beta.
\]
This is possible since $\beta<\frac{1}{3}-2\kappa$.
Absorbing the logarithmic factors into $x^{o(1)}$ completes the proof of \eqref{eq:Px-power-saving-claim}.
\end{proof}

\section{Proof of Theorem~\ref{thm:intro:first}}\label{sec:main1}

In this section, we give the proof of Theorem~\ref{thm:intro:first}. We begin with some notations. For integers $m\ge 0$, we define
\[
N_m(x):=\#\{p\le x:\ f_{\varepsilon}(p)>m\}
\qquad \text{and} \qquad
a_m(x):=\frac{N_m(x)}{\piodd(x)}.
\]
Using the double counting trick, we have
\begin{equation}\label{eq:avgFx}
\frac{1}{\piodd(x)}\sum_{p\le x} f_\eps(p)=\sum_{m\ge 0} a_m(x).
\end{equation}

\subsection{Finite-dimensional equidistribution and limiting tail probabilities}
The identity~\eqref{eq:avgFx} shows that it suffices to understand the tail probabilities $a_m(x)$.
For fixed $m$, the condition $f_\varepsilon(p)>m$ only depends on finitely many Legendre symbols
$\chi_p(q)$ with $q\le m$, so we begin by recording their equidistribution.


\begin{lemma}\label{lem:equidistribution}
Let $q_1,\dots,q_r$ be distinct primes and let $\eta_1,\dots,\eta_r\in\{\pm 1\}$. Then
\[
\#\Bigl\{p\le x:\ p\notin\{q_1,\dots,q_r\},\ \chi_p(q_i)=\eta_i\ \ (1\le i\le r)\Bigr\}
=2^{-r}\pi(x)+o(\pi(x)).
\]
\end{lemma}

\begin{proof}
For odd primes $p\notin\{q_1,\dots,q_r\}$ we define
\[
J(p):=\prod_{i=1}^r\frac{1+\eta_i \cdot \chi_p(q_i)}{2}.
\]
It is easy to see that $J(p)=1$ if $\chi_p(q_i)=\eta_i$ for all $i$ and $J(p)=0$ otherwise.
Expanding the product gives
\[
J(p)=2^{-r}\sum_{S\subseteq\{1,\dots,r\}}\Bigl(\prod_{i\in S}\eta_i\Bigr)\chi_p(q_S),
\]
where $q_S=\prod_{i\in S}q_i$. For $S=\emptyset$ the summand is $1$. For $S\neq\emptyset$, the map
$p\mapsto \chi_p(q_S)$ is a nonprincipal quadratic Dirichlet character in the
prime variable $p$. By the prime number
theorem for Dirichlet characters \cite[Corollary~5.29]{IK}, we have
\[
\sum_{p\le x}\chi_p(q_S)=o(\pi(x)).
\]
Summing $J(p)$ over primes $p\le x$ and discarding the finitely many primes $p\in\{q_1,\dots,q_r\}$ gives
\[
\sum_{p\le x} J(p)=2^{-r}\pi(x)+o(\pi(x)),
\]
which is exactly the desired counting statement.
\end{proof}


Fix $m\ge 1$, and list the primes $\le m$ as
\[
q_1<q_2<\cdots<q_{\pi(m)}\qquad(\text{so } q_1=2 \text{ when } m\ge 2).
\]
For an odd prime $p>m$ define
\[
v_m(p):=\left(\chi_p(q_1),\chi_p(q_2),\dots,\chi_p(q_{\pi(m)})\right)\in\{\pm1\}^{\pi(m)}.
\]

\begin{lemma}\label{lem:qm-exists}
Fix $m\ge 1$. Let
\[
A_m=\Bigl\{\sigma\in \{\pm 1\}^{\pi(m)}:\ v_m(p)=\sigma \text{ and } f_{\varepsilon}(p)>m \text{ for some odd prime } p>m\Bigr\}.
\]
Then the limit
\begin{equation}\label{eq:qm-def}
\widehat{a}_m:=\lim_{x\to\infty}a_m(x)
\end{equation}
exists and equals $\frac{|A_m|}{2^{\pi(m)}}$.
\end{lemma}
\begin{proof}
We first show that for every odd prime $p>m$,
\[
f_{\varepsilon}(p)>m\iff v_m(p)\in A_m.
\]
Indeed, if $p,p'>m$ are odd primes with $v_m(p)=v_m(p')$, then $\chi_p(q)=\chi_{p'}(q)$ for every prime $q\le m$.
By complete multiplicativity, this implies $\chi_p(n)=\chi_{p'}(n)$ for all integers $1\le n\le m$, and hence
$S_k(p)=S_k(p')$ for every $k\le m$. In particular, $f_{\varepsilon}(p)>m$ if and only if $f_{\varepsilon}(p')>m$.

Now if $v_m(p)\in A_m$, then by definition of $A_m$, we see that there exists an odd prime $p'>m$ such that $v_m(p')=v_m(p)$ and $f_{\varepsilon}(p')>m$. By the discussion above, we have $f_{\varepsilon}(p)>m$. Conversely, if $f_{\varepsilon}(p)>m$, then $v_m(p)\in A_m$ by definition.

Therefore,
\[
N_m(x)=\sum_{\sigma\in A_m}\#\{p\le x:\ p>m,\ v_m(p)=\sigma\}+O(1),
\]
where the $O(1)$ term accounts for the finitely many primes $p\le m$.
Applying Lemma~\ref{lem:equidistribution}  to each fixed
$\sigma\in\{\pm1\}^{\pi(m)}$ yields
\[
\#\{p\le x:\ p>m,\ v_m(p)=\sigma\}
=\left(2^{-\pi(m)}+o(1)\right)\pi(x).
\]
Summing over $\sigma\in A_{m}$ and dividing by $\piodd(x)$ gives
\[
a_{m}(x)=\frac{|A_{m}|}{2^{\pi(m)}}+o(1),
\]
so the limit exists and equals $|A_{m}|/2^{\pi(m)}$.
\end{proof}

\subsection{Uniform tail bounds and completion of the proof}

\begin{proposition}\label{prop:UI}
We have
\[
\lim_{M\to\infty}\ \limsup_{x\to\infty}\ \sum_{m>M} a_m(x)=0.
\]
\end{proposition}

\begin{proof}
Fix $\kappa\in(0,1/6)$ and choose $\delta\in(0,1)$ satisfying the condition given in the proof of Theorem \ref{lem:Px-large-o1}. For any fixed $M$, take $x$ sufficiently large so that
$M<x^{1/6-\kappa}$ and thus
\[
\sum_{m>M} a_m(x)=\sum_{M<m\le x^{1/6-\kappa}} a_m(x)\;+\;\sum_{m>x^{1/6-\kappa}} a_m(x).
\]
Note that if $f_\eps(p)>m$, then $S_m(p)\ge \eps m$. So we have a trivial inequality $a_m(x)\leq P_m(x)$. By \eqref{eq:am-smallm}, the first sum is bounded by
\[
\sum_{M<m\le x^{1/6-\kappa}} a_m(x)\ll_{\eps,\kappa,\delta}\sum_{m>M} m^{-3+\delta}.
\]
The second sum is $o(1)$ by Theorem~\ref{lem:Px-large-o1}. Hence
\[
\limsup_{x\to\infty}\sum_{m>M} a_m(x)\ll_{\eps,\kappa,\delta}\sum_{m>M} m^{-3+\delta}.
\]
Letting $M\to\infty$ gives the claim.
\end{proof}

We are now ready to present the following.

\begin{proof}[Proof of Theorem~\ref{thm:intro:first}]
We first show that $\sum_{m\ge 0} \widehat{a}_m<\infty$. In fact, we fix suitable constants $\kappa\in(0,1/6)$ and $\delta\in(0,1)$. For each fixed $m>0$, taking $x$ sufficiently large so that
$m\le x^{1/6-\kappa}$ and using \eqref{eq:am-smallm} with the trivial inequality $a_m(x)\leq P_m(x)$, we obtain
\[
\widehat{a}_m=\lim_{x\to\infty}a_m(x)\le \sup_{x\ge m^{6/(1-6\kappa)}} a_m(x)\ll_{\eps,\kappa,\delta} m^{-3+\delta}.
\]
Then the convergence of the series $\sum_{m\ge 0} \widehat{a}_m$ follows. We write $c_\eps :=\sum_{m\ge 0} \widehat{a}_m$.

Fix $M\ge 0$. Splitting \eqref{eq:avgFx} gives
\[
\frac{1}{\piodd(x)}\sum_{p\le x} f_\eps(p)=\sum_{m=0}^{M} a_m(x)\;+\;\sum_{m>M} a_m(x).
\]
Taking $\liminf_{x\to\infty}$ and using nonnegativity yields
\[
\liminf_{x\to\infty}\frac{1}{\piodd(x)}\sum_{p\le x} f_\eps(p)\ge \sum_{m=0}^{M} \widehat{a}_m.
\]
Letting $M\to\infty$ gives
\[
\liminf_{x\to\infty}\frac{1}{\piodd(x)}\sum_{p\le x} f_\eps(p)\ge \sum_{m\ge 0} \widehat{a}_m=c_\eps.
\]

For the $\limsup$, Proposition~\ref{prop:UI} implies that for every $\eta>0$ there exists $M$ such that
$\limsup_{x\to\infty}\sum_{m>M} a_m(x)\le \eta$. Hence
\[
\limsup_{x\to\infty}\frac{1}{\piodd(x)}\sum_{p\le x} f_\eps(p)
\le \lim_{x\to\infty}\sum_{m=0}^{M} a_m(x)+\eta
= \sum_{m=0}^{M} \widehat{a}_m+\eta
\le c_\eps+\eta.
\]
Since $\eta>0$ is arbitrary, we get $\limsup\le c_\eps$. Combining with the $\liminf$ gives
\[
\lim_{x\to\infty}\frac{1}{\piodd(x)}\sum_{p\le x} f_\eps(p)=c_{\varepsilon}.
\]
This implies that $\sum_{p\le x} f_\eps(p)\sim c_\eps\,\frac{x}{\log x}$ by the prime number theorem.
\end{proof}

\section*{Acknowledgements}
The authors thank Terence Tao for valuable comments on the Erd\H{o}s Problems website,
which significantly improved an earlier version of this paper. They thank Igor Shparlinski
for pointing out that \cite{BKS21} proves a stronger result than Lemma~\ref{lem:C3m-bound}. They are also grateful to Thomas Bloom for founding and maintaining the Erd\H{o}s Problems website.


\begin{thebibliography}{99}

\bibitem{EP981}
T. F. Bloom, Erd\H{o}s Problem \#981, \url{https://www.erdosproblems.com/981}, accessed 2025-12-22.


\bibitem{BKS21}
R.~de~la~Bret\`eche, P.~Kurlberg, and I.~E.~Shparlinski,
On the number of products which form perfect powers and discriminants of multiquadratic extensions,
\emph{Int.\ Math.\ Res.\ Not.} \textbf{2021} (2021), no.~22, 17140--17169.




\bibitem{Er65b}
P. Erd\H{o}s, Some recent advances and current problems in number theory, \emph{Lectures on Modern Mathematics}, Vol. III (1965), 196--244.



\bibitem{El69}
P.~D.~T.~A.~Elliott, A conjecture of Erd\H{o}s concerning character sums,
\emph{Nederl.\ Akad.\ Wetensch.\ Proc.\ Ser.\ A} \textbf{72} = \emph{Indag.\ Math.} \textbf{31} (1969), 164--171.


\bibitem{HW}
G.~H.~Hardy and E.~M.~Wright, \emph{An Introduction to the Theory of Numbers}, 6th ed., Oxford University Press, 2008.

\bibitem{HB95}
D. R. Heath-Brown, A mean value estimate for real character sums, \emph{Acta Arith.} \textbf{72(3)} (1995), 235--275.

\bibitem{HL}
A.~Hussain and Y.~Lamzouri, The limiting distribution of Legendre paths, \emph{J. \'{E}c. polytech. Math.} \textbf{11} (2024), 589--611.

\bibitem{IK}
H. Iwaniec and E. Kowalski, \emph{Analytic Number Theory}, Amer. Math. Soc. Colloq. Publ., Vol. 53, AMS, Providence, RI, 2004.

\bibitem{La24}
Y.~Lamzouri, The distribution of large quadratic character sums and applications, \emph{Algebra \& Number Theory} \textbf{18} (2024), no.~11, 2091--2131.

\bibitem{MoVa79}
H.~L. Montgomery and R.~C. Vaughan, Mean values of character sums, \emph{Canad. J. Math.} \textbf{31} (1979), 476--487.

\bibitem{MV}
H.~L.~Montgomery and R.~C.~Vaughan, \emph{Multiplicative Number Theory I: Classical Theory}, Cambridge Studies in Advanced Mathematics \textbf{97}, Cambridge University Press, 2007.

\end{thebibliography}
\end{document}